\documentclass{article}

\usepackage{amsmath,amssymb,amsthm}

\newtheorem{dfn}{Definition}
\newtheorem{thm}[dfn]{Theorem}
\newtheorem{cor}[dfn]{Corollary}
\newtheorem{ex}[dfn]{Example}

\title{Chain rules and inequalities for the BHT
fractional calculus on arbitrary time scales\thanks{This is 
a preprint of a paper whose final and definite form is published in 
\emph{Arabian Journal of Mathematics}, ISSN 2193-5343 (Print) 
2193-5351 (Online). Paper submitted 25/Feb/2016; 
revised 24/Sept/2016; accepted 28/Nov/2016.}}

\author{Eze R. Nwaeze$^1$\\
{\tt enwaeze@mytu.tuskegee.edu}
\and
Delfim F. M. Torres$^2$\\
{\tt delfim@ua.pt}}

\date{$^1$Department of Mathematics, Tuskegee University,\\
Tuskegee, AL 36088, USA\\[0.3cm]
$^2$CIDMA, Department of Mathematics, University of Aveiro,\\
3810-193 Aveiro, Portugal}

\begin{document}

\maketitle

\begin{abstract}
We develop the Benkhettou--Hassani--Torres
fractional (noninteger order)
calculus on time scales by proving
two chain rules for the
$\alpha$-fractional derivative
and five inequalities
for the $\alpha$-fractional integral.
The results coincide with well-known
classical results when the operators are
of (integer) order $\alpha = 1$
and the time scale coincides with
the set of real numbers.

\bigskip

\noindent \textbf{Keywords:} local fractional calculus;
calculus on time scales; chain rules; integral inequalities.

\medskip

\noindent \textbf{MSC 2010:} 26A33; 26D10; 26E70.
\end{abstract}


\section{Introduction}

The study of fractional (noninteger order)
calculus on time scales is a subject of strong
current interest \cite{MyID:152,MyID:201,MyID:296,MyID:320}.
Recently, Benkhettou, Hassani and Torres
introduced a (local) fractional calculus on arbitrary
time scales $\mathbb{T}$ (called here the BHT fractional
calculus) based on the $T_\alpha$ differentiation operator
and the $\alpha$-fractional integral \cite{MyID:324}.
The Hilger time-scale calculus \cite{BookTS:2001}
is then obtained as a particular case, by choosing $\alpha=1$.
In this paper we develop the BHT time-scale fractional calculus
initiated in \cite{MyID:324}. Precisely, we prove
two different chain rules for the fractional derivative $T_\alpha$
(Theorems~\ref{CR} and \ref{thm:CR2}) and several
inequalities for the $\alpha$-fractional integral:
H\"{o}lder's inequality (Theorem~\ref{thm:Hineq}),
Cauchy--Schwarz's inequality (Theorem~\ref{thm:CSineq}),
Minkowski's inequality (Theorem~\ref{thm:Mink_ineq}),
generalized Jensen's fractional inequality (Theorem~\ref{thm:Jen})
and a weighted fractional Hermite--Hadamard inequality 
on time scales (Theorem~\ref{thm:HHineq}).

The paper is organized as follows. In Section~\ref{sec:Prelim}
we recall the basics of the the BHT fractional calculus.
Our results are then formulated and proved in Section~\ref{sec:MR}.


\section{Preliminaries}
\label{sec:Prelim}

We briefly recall the necessary notions
from the BHT fractional calculus \cite{MyID:324}:
fractional differentiation and fractional integration
on time scales. For an introduction to the time-scale
theory we refer the reader to the book \cite{BookTS:2001}.

\begin{dfn}[See \cite{MyID:324}]
\label{def:fd:ts}
Let $\mathbb{T}$ be a time scale,
$f:\mathbb{T}\rightarrow \mathbb{R}$, $t\in \mathbb{T}^{\kappa}$,
and $\alpha \in (0,1]$. For $t>0$, we define
$T_\alpha(f)(t)$ to be the number (provided it exists) with the property
that, given any $\epsilon >0$, there is a $\delta$-neighbourhood
$\mathcal{V}_t = \left(t-\delta ,t+\delta\right) \cap \mathbb{T}$
of $t$, $\delta > 0$, such that
$\left \vert \left[ f(\sigma (t))-f(s)\right]t^{1-\alpha}
-T_\alpha(f)(t)\left[ \sigma(t)-s\right]\right \vert
\leq \epsilon \left \vert \sigma
(t)-s\right \vert$
for all $s\in \mathcal{V}_t$. We call
$T_\alpha(f)(t)$ the $\alpha$-fractional derivative of $f$
of order $\alpha $ at $t$, and we define the
$\alpha$-fractional derivative at 0 as
$T_\alpha(f)(0):=\displaystyle\lim_{t\rightarrow 0^+} T_\alpha(f)(t)$.
\end{dfn}

If $\alpha = 1$, then we obtain from Definition~\ref{def:fd:ts}
the Hilger delta derivative of time scales \cite{BookTS:2001}.
The $\alpha$-fractional derivative of order zero
is defined by the identity operator: $T_0(f) := f$.
The basic properties of the $\alpha$-fractional derivative
on time scales are given in \cite{MyID:324}, together with
several illustrative examples. Here we just recall the
item (iv) of Theorem 4 in \cite{MyID:324}, which is needed
in the proof of our Theorem~\ref{CR}. 

\begin{thm}[See \cite{MyID:324}]
Let $\alpha\in (0, 1]$ and $\mathbb{T}$ be a time scale. 
Assume $f: \mathbb{T}\rightarrow \mathbb{R}$ and let 
$t\in {\mathbb{T}}^\kappa$. If $f$ is $\alpha$-fractional 
differentiable of order $\alpha$ at $t$, then 
$$
f(\sigma(t))=f(t)+\mu(t)t^{\alpha-1}T_{\alpha}(f)(t).
$$
\end{thm}

The other main operator of \cite{MyID:324} is
the $\alpha$-fractional integral of
$f : \mathbb{T} \rightarrow \mathbb{R}$, defined by
$$
\int f(t) \Delta^\alpha t := \int f(t) t^{\alpha-1} \Delta t,
$$
where the integral on the right-hand side is the usual
Hilger delta-integral of time scales \cite[Def.~26]{MyID:324}.
If $F_{\alpha}(t):=\int f(t)\Delta^{\alpha} t$,
then one defines the Cauchy $\alpha$-fractional integral by
$\int_{a}^{b}f(t)\Delta^{\alpha} t
:=F_{\alpha}(b)-F_{\alpha}(a)$, where $a, b\in \mathbb{T}$
\cite[Def.~28]{MyID:324}. The interested reader
can find the basic properties of the Cauchy
$\alpha$-fractional integral in \cite{MyID:324}. 
Here we are interested to prove some fractional
integral inequalities on time scales. For that, we use
some of the properties of \cite[Theorem~31]{MyID:324}.

\begin{thm}[Cf. Theorem~31 of \cite{MyID:324}]
\label{Int-Proprty}
Let $\alpha\in (0,\ 1]$, $a, b, c \in \mathbb{T}$, $\gamma\in\mathbb{R}$,
and $f, g$ be two rd-continuous functions. Then,
\begin{enumerate}

\item[(i)] $\int_{a}^{b}[f(t)+g(t)]\Delta^{\alpha} t
= \int_{a}^{b}f(t)\Delta^{\alpha} t 
+ \int_{a}^{b}g(t)\Delta^{\alpha} t$;

\item[(ii)] $\int_{a}^{b}(\gamma f)(t)\Delta^{\alpha} t
= \gamma \int_{a}^{b}f(t)\Delta^{\alpha} t$;

\item[(iii)] $\int_{a}^{b}f(t)\Delta^{\alpha} t
= - \int_{b}^{a}f(t)\Delta^{\alpha} t$;

\item[(iv)] $\int_{a}^{b}f(t)\Delta^{\alpha} t
= \int_{a}^{c}f(t)\Delta^{\alpha} t 
+\int_{c}^{b}f(t)\Delta^{\alpha} t$;

\item[(v)] if there exist $g: \mathbb{T}\rightarrow \mathbb{R}$
with $|f(t)|\leq g(t)$ for all $t\in [a,\ b]$, then
$\left| \int_{a}^{b} f(t)\Delta^{\alpha} t\right|
\leq\int_{a}^{b} g(t)\Delta^{\alpha} t$.
\end{enumerate}
\end{thm}


\section{Main Results}
\label{sec:MR}

The chain rule, as we know it from the classical
differential calculus, does not hold for the
BHT fractional calculus. A simple example of this
fact has been given in \cite[Example~20]{MyID:324}.
Moreover, it has been shown in \cite[Theorem~21]{MyID:324},
using the mean value theorem, that if
$g:\mathbb{T}\rightarrow\mathbb{R}$ is continuous
and fractional differentiable of order $\alpha \in (0, 1]$
at $t \in \mathbb{T}^\kappa$ and
$f:\mathbb{R}\rightarrow\mathbb{R}$
is continuously differentiable, then there exists
$c \in [t, \sigma(t)]$ such that
$T_\alpha(f\circ g)(t) = f'(g(c)) T_\alpha(g)(t)$.
In Section~\ref{sec:3.1}, we provide two other chain rules.
Then, in Section~\ref{sec:3.2}, we prove some fractional
integral inequalities on time scales.


\subsection{Fractional chain rules on time scales}
\label{sec:3.1}

\begin{thm}[Chain Rule I]
\label{CR}
Let $f:\mathbb{R}\rightarrow\mathbb{R}$ be continuously differentiable,
$\mathbb{T}$ be a given time scale and $g:\mathbb{T}\rightarrow\mathbb{R}$
be $\alpha$-fractional differentiable. Then,
$f\circ g: \mathbb{T}\rightarrow\mathbb{R}$ is also
$\alpha$-fractional differentiable with
\begin{equation}
\label{eq:CR}
T_\alpha(f\circ g)(t)=\Bigg[\int_0^1 f'\left(g(t)
+h\mu(t)t^{\alpha-1}T_\alpha(g)(t)\right)dh\Bigg]
T_\alpha(g)(t).
\end{equation}
\end{thm}

\begin{proof}
We begin by applying the ordinary substitution rule from calculus:
\begin{align*}
f(g(\sigma(t)))-f(g(s))
&= \int_{g(s)}^{g(\sigma(t))}f'(\tau)d\tau\\
&=[g(\sigma(t))-g(s)]\int_0^1 f'(hg(\sigma(t))+(1-h)g(s))dh.
\end{align*}
Let $t\in {\mathbb{T}}^\kappa$ and $\epsilon>0$.
Since $g$ is $\alpha$-fractional differentiable at $t$,
we know from Definition~\ref{def:fd:ts} that
there exists a neighbourhood $U_1$ of $t$ such that
$$
\left|[g(\sigma(t)) - g(s)]t^{1-\alpha} 
- T_\alpha(g)(t) (\sigma(t)- s)\right|
\leq \epsilon^{*}|\sigma(t)-s|
\quad \text{ for all } s \in U_1,
$$
where
$$
\epsilon^{*} = \dfrac{\epsilon}{\displaystyle
1+2\int_0^1 \left|f'(hg(\sigma(t))+(1-h)g(t))\right|dh}.
$$
Moreover, $f'$ is continuous on $\mathbb{R}$ and, therefore,
it is uniformly continuous on closed subsets of $\mathbb{R}$.
Observing that $g$ is also continuous,
because it is $\alpha$-fractional differentiable
(see item (i) of Theorem~4 in \cite{MyID:324}),
there exists a neighbourhood $U_2$ of $t$ such that
$$
|f'(hg(\sigma(t))+(1-h)g(s)) - f'(hg(\sigma(t))+(1-h)g(t))|
\leq \frac{\epsilon}{2(\epsilon^{*}+|T_{\alpha}(g)(t)|)}
$$
for all $s\in U_2$. To see this, note that
\begin{align*}
|hg(\sigma(t))+(1-h)g(s)- (hg(\sigma(t))+(1-h)g(t))|&= (1-h)|g(s)-g(t)|\\
&\leq |g(s)-g(t)|
\end{align*}
holds for all $0\leq h\leq 1$. We then define
$U:=U_1\cap U_2$ and let $s\in U$. For convenience, we put
$$
\gamma = hg(\sigma(t))+(1-h)g(s) 
\quad \text{ and } \quad
\beta = hg(\sigma(t))+(1-h)g(t).
$$
Then we have
\begin{align*}
&\Bigg|[(f\circ g)(\sigma(t)) - (f\circ g)(s)]t^{1-\alpha}
- T_\alpha(g)(t) (\sigma(t)- s)\int_0^1 f'(\beta)dh\Bigg|\\
&=\Bigg|t^{1-\alpha}[g(\sigma(t))-g(s)]\int_0^1 f'(\gamma)dh
-  T_\alpha(g)(t) (\sigma(t)- s)\int_0^1 f'(\beta)dh \Bigg|\\
&=\Bigg|\Big(t^{1-\alpha}[g(\sigma(t))-g(s)]
- (\sigma(t)-s)T_{\alpha}(g)(t)  \Big)\\
& \quad \times \int_0^1 f'(\gamma)dh
+ T_\alpha(g)(t) (\sigma(t)- s)\int_0^1 (f'(\gamma)-f'(\beta))dh\Bigg|\\
&\leq \Big|t^{1-\alpha}[g(\sigma(t))-g(s)]
- (\sigma(t)-s)T_{\alpha}(g)(t)  \Big|\int_0^1 |f'(\gamma)|dh \\
&\quad + \big|T_\alpha(g)(t)\big|| \sigma(t)- s|\int_0^1 |f'(\gamma)-f'(\beta)|dh\\
&\leq \epsilon^{*}|\sigma(t)-s| \int_0^1 |f'(\gamma)|dh
+ \big[\epsilon^{*}+ \big|T_\alpha(g)(t)\big|\big]| \sigma(t)
- s|\int_0^1 |f'(\gamma)-f'(\beta)|dh\\
&\leq \frac{\epsilon}{2}|\sigma(t)-s| + \frac{\epsilon}{2}|\sigma(t)-s|\\
&=\epsilon|\sigma(t)-s|.
\end{align*}
Therefore, $f\circ g$ is $\alpha$-fractional differentiable at $t$
and \eqref{eq:CR} holds.
\end{proof}

Let us illustrate Theorem~\ref{CR} with an example.

\begin{ex}
Let $g:\mathbb{Z}\rightarrow\mathbb{R}$
and $f:\mathbb{R}\rightarrow\mathbb{R}$
be defined by
$$
g(t)=t^2 \text{ and } f(t)=e^{t}.
$$
Then,
$T_{\alpha}(g)(t)=(2t+1)t^{1-\alpha} \text{ and } f'(t)=e^t$.
Hence, we have by Theorem~\ref{CR} that
\begin{align*}
T_\alpha(f\circ g)(t)
&=\Bigg[\int_0^1 f'(g(t)+h\mu(t)t^{\alpha-1}T_\alpha(g)(t))dh\Bigg]T_\alpha(g)(t)\\
&= (2t+1)t^{1-\alpha}\int_0^1 e^{t^2+h(2t+1)}dh\\
&=(2t+1)t^{1-\alpha}e^{t^2}\int_0^1 e^{h(2t+1)}dh\\
&=(2t+1)t^{1-\alpha}e^{t^2}\frac{1}{2t+1}\big[ e^{2t+1}-1\big]\\
&=t^{1-\alpha}e^{t^2}\big[ e^{2t+1} - 1\big].
\end{align*}
\end{ex}

\begin{thm}[Chain Rule II]
\label{thm:CR2}
Let $\mathbb{T}$ be a time scale.
Assume $\nu:\mathbb{T}\rightarrow \mathbb{R}$ is strictly increasing
and $\tilde{\mathbb{T}}:=\nu(\mathbb{T})$ is also a time scale.
Let $w:\tilde{\mathbb{T}}\rightarrow \mathbb{R}$, $\alpha\in (0, 1]$,
and $\tilde{T_{\alpha}}$ denote the $\alpha$-fractional derivative
on $\tilde{\mathbb{T}}$. If for each $t\in {\mathbb{T}}^\kappa$,
$\tilde{T_{\alpha}}(w)(\nu(t))$ exists and for every $\epsilon>0$, 
there is a neighbourhood $V$ of $t$ such that
$$
|\tilde{\sigma}(\nu(t))-\nu(s) - T_\alpha(\nu)(t) (\sigma(t)- s)|
\leq \epsilon|\sigma(t)-s| 
\quad \text{ for all } s\in V,
$$
where $\tilde{\sigma}$ denotes the forward jump operator
on $\tilde{\mathbb{T}}$, then
$$
T_{\alpha}(w\circ \nu)(t)
=\big[\tilde{T_{\alpha}}(w)\circ \nu\big](t)
T_{\alpha}(\nu)(t).
$$
\end{thm}

\begin{proof}
Let $0<\epsilon<1$ be given and define
$\epsilon^{*}:=\epsilon\Big[1+ \big|T_{\alpha}(\nu)(t)\big|
+ \big| \tilde{T_{\alpha}}(w)(\nu(t))\big|  \Big]^{-1}$.
Note that $0<\epsilon^{*}<1$. According to the assumptions,
there exist neighbourhoods $U_1$ of $t$ and $U_2$ of $\nu(t)$ such that
$$
|\tilde{\sigma}(\nu(t))-\nu(s) - T_\alpha(\nu)(t) (\sigma(t)- s)|
\leq \epsilon^{*}|\sigma(t)-s|
$$
for all $s\in U_1$ and
$$
\big|[w(\tilde{\sigma}(\nu(t))) - w(r)]t^{1-\alpha}
- \tilde{T_{\alpha}}(w)(\nu(t)) (\tilde{\sigma}(\nu(t))- r)\big|
\leq \epsilon^{*}|\tilde{\sigma}(\nu(t))-r|
$$
for all $r\in U_2$. Let $U := U_1\cap \nu^{-1}(U_2)$. 
For any $s\in U$, we have that $s\in U_1$ and $\nu(s)\in U_2$ with
\begin{align*}
\Big|&[w(\nu(\sigma(t))) - w(\nu(s))]t^{1-\alpha}
- (\sigma(t)- s) \big[\tilde{T_{\alpha}}(w)(\nu(t))\big]
T_{\alpha}(\nu)(t)\Big|\\
&=\Big|[w(\nu(\sigma(t))) - w(\nu(s))]t^{1-\alpha}
-[\tilde{\sigma}(\nu(t))-\nu(s)]\tilde{T_{\alpha}}(w)(\nu(t))\\
&\quad + [\tilde{\sigma}(\nu(t))-\nu(s)
- T_\alpha(\nu)(t) (\sigma(t)- s)]\tilde{T_{\alpha}}(w)(\nu(t))\Big|\\
&\leq \epsilon^{*}|\tilde{\sigma}(\nu(t))-\nu(s)|
+  \epsilon^{*}|\sigma(t)-s||\tilde{T_{\alpha}}(w)(\nu(t))|\\
&\leq \epsilon^{*}\Big[|\tilde{\sigma}(\nu(t))-\nu(s)
-(\sigma(t)-s)T_{\alpha}(\nu)(t)|\\
&\quad + |\sigma(t)-s||T_{\alpha}(\nu)(t)|
+  |\sigma(t)-s||\tilde{T_{\alpha}}(w)(\nu(t))|\Big]\\
\end{align*}
\begin{align*}
&\leq \epsilon^{*}\Big[\epsilon^{*}|\sigma(t)-s|
+|\sigma(t)-s||T_{\alpha}(\nu)(t)|
+ |\sigma(t)-s||\tilde{T_{\alpha}}(w)(\nu(t))| \Big]\\
&= \epsilon^{*}|\sigma(t)-s|\Big[\epsilon^{*}+|T_{\alpha}(\nu)(t)|
+ |\tilde{T_{\alpha}}(w)(\nu(t))| \Big]\\
&\leq \epsilon^{*}\Big[1+|T_{\alpha}(\nu)(t)|
+ |\tilde{T_{\alpha}}(w)(\nu(t))| \Big]|\sigma(t)-s|\\
&=\epsilon |\sigma(t)-s|.
\end{align*}
This proves the claim.
\end{proof}


\subsection{Fractional integral inequalities on time scales}
\label{sec:3.2}

The $\alpha$-fractional integral on time scales
was introduced in \cite[Section~3]{MyID:324},
where some basic properties were proved.
Here we show that the $\alpha$-fractional integral
satisfies appropriate fractional versions of the
fundamental inequalities of H\"{o}lder,
Cauchy--Schwarz, Minkowski, Jensen and Hermite--Hadamard.

\begin{thm}[H\"{o}lder's fractional inequality on time scales]
\label{thm:Hineq}
Let $\alpha\in (0, 1]$ and $a, b\in\mathbb{T}$.
If $f, g, h :[a, b]\rightarrow\mathbb{R}$
are $rd$-continuous, then
\begin{equation}
\label{eq:Hineq}
\int_a^b |f(t)g(t)| |h(t)| \Delta^{\alpha}t
\leq \Bigg[\int_a^b |f(t)|^p |h(t)| \Delta^{\alpha}t\Bigg]^{\frac{1}{p}}\Bigg[
\int_a^b |g(t)|^q |h(t)| \Delta^{\alpha}t\Bigg]^{\frac{1}{q}},
\end{equation}
where $p>1$ and $\frac{1}{p} + \frac{1}{q} = 1$.
\end{thm}

\begin{proof}
For nonnegative real numbers $A$ and $B$, the basic inequality
$$
A^{1/p}B^{1/q} \leq \frac{A}{p}+\frac{B}{q}
$$
holds. Now, suppose, without loss of generality, that
$$
\Bigg[\int_a^b |f(t)|^p |h(t)|\Delta^{\alpha}t\Bigg]
\Bigg[\int_a^b |g(t)|^q |h(t)|\Delta^{\alpha}t\Bigg]
\neq 0.
$$
Applying Theorem~\ref{Int-Proprty} and the above inequality to
$$
A(t)=\frac{|f(t)|^p |h(t)|}{\int_a^b|f(\tau)|^p |h(\tau)|\Delta^{\alpha}\tau}
$$
and
$$
B(t)=\frac{|g(t)|^q |h(t)|}{\int_a^b|g(\tau)|^p |h(\tau)| \Delta^{\alpha}\tau},
$$
and integrating the obtained inequality between $a$ and $b$,
which is possible since all occurring functions are $rd$-continuous,
we find that
\begin{align*}
\int_a^b & [A(t)]^{1/p}[B(t)]^{1/q}\Delta^{\alpha}t\\
&= \int_a^b \frac{|f(t)| |h(t)|^{1/p}}{\Big[\int_a^b|f(\tau)|^p|h(\tau)|\Delta^{\alpha}\tau\Big]^{1/p}}
\frac{|g(t)| |h(t)|^{1/q}}{\Big[\int_a^b|g(\tau)|^q |h(\tau)|\Delta^{\alpha}\tau\Big]^{1/q}}\Delta^{\alpha}t\\
&\leq \int_a^b\Bigg[\frac{A(t)}{p}+\frac{B(t)}{q}\Bigg]\Delta^{\alpha}t\\
&=\int_a^b\Bigg[\frac{1}{p}\frac{|f(t)|^p |h(t)|}{\int_a^b|f(\tau)|^p |h(\tau)|\Delta^{\alpha}\tau}
+\frac{1}{q}\frac{|g(t)|^q |h(t)|}{\int_a^b|g(\tau)|^q |h(\tau)|\Delta^{\alpha}\tau}\Bigg]\Delta^{\alpha}t\\
&=\frac{1}{p}\int_a^b\Bigg[\frac{|f(t)|^p |h(t)|}{\int_a^b
|f(\tau)|^p |h(\tau)|\Delta^{\alpha}\tau}\Bigg]\Delta^{\alpha}t
+\frac{1}{q}\int_a^b\Bigg[\frac{|g(t)|^q |h(t)|}{\int_a^b
|g(\tau)|^q |h(\tau)|\Delta^{\alpha}\tau}\Bigg]\Delta^{\alpha}t\\
&\leq \frac{1}{p}+\frac{1}{q}\\
&=1.
\end{align*}
This directly yields the H\"{o}lder inequality \eqref{eq:Hineq}.
\end{proof}

As a particular case of Theorem~\ref{thm:Hineq},
we obtain the following inequality.

\begin{thm}[Cauchy--Schwarz's fractional inequality on time scales]
\label{thm:CSineq}
Let $\alpha\in (0, 1]$ and $a, b\in\mathbb{T}$.
If $f, g, h :[a, b]\rightarrow\mathbb{R}$ are $rd$-continuous, then
$$
\int_a^b |f(t)g(t)| |h(t)| \Delta^{\alpha}t
\leq \sqrt{\Bigg[\int_a^b |f(t)|^2 |h(t)| \Delta^{\alpha}t\Bigg]
\Bigg[\int_a^b |g(t)|^2 |h(t)| \Delta^{\alpha}t\Bigg]}.
$$
\end{thm}

\begin{proof}
Choose $p=q=2$ in H\"{o}lder's inequality \eqref{eq:Hineq}.
\end{proof}

Using H\"{o}lder's inequality \eqref{eq:Hineq}, we can also
prove the following result.

\begin{cor}
Let $\alpha\in (0, 1]$ and $a, b\in\mathbb{T}$.
If $f, g, h :[a, b]\rightarrow\mathbb{R}$ are $rd$-continuous, then
$$
\int_a^b |f(t)g(t)| |h(t)| \Delta^{\alpha}t
\geq \Bigg[\int_a^b |f(t)|^p |h(t)| \Delta^{\alpha}t\Bigg]^{\frac{1}{p}}
\Bigg[\int_a^b |g(t)|^q |h(t)| \Delta^{\alpha}t\Bigg]^{\frac{1}{q}},
$$
where $\frac{1}{p}+\frac{1}{q}=1$ and $p < 0$ or $q < 0$.
\end{cor}

\begin{proof}
Without loss of generality, we may assume that $p < 0$ and $q > 0$.
Set $P = - \frac{p}{q}$ and $Q = \frac{1}{q}$. Then,
$\frac{1}{P}+\frac{1}{Q}=1$ with $P > 1$ and $Q > 0$. 
From \eqref{eq:Hineq} we can write that
\begin{multline}
\label{eq:Hineq:FG}
\int_a^b |F(t)G(t)| |h(t)| \Delta^{\alpha}t\\
\leq \Bigg[\int_a^b |F(t)|^P |h(t)| \Delta^{\alpha}t\Bigg]^{\frac{1}{P}}\Bigg[
\int_a^b |G(t)|^Q |h(t)| \Delta^{\alpha}t\Bigg]^{\frac{1}{Q}}
\end{multline}
for any $rd$-continuous functions $F, G:[a, b]\rightarrow\mathbb{R}$. 
The desired result is obtained by taking 
$F(t) = [f(t)]^{-q}$ and $G(t) = [f(t)]^{q} [g(t)]^{q}$
in inequality \eqref{eq:Hineq:FG}.
\end{proof}

Next, we use H\"{o}lder's inequality \eqref{eq:Hineq}
to deduce a fractional Minkowski's inequality on time scales.

\begin{thm}[Minkowski's fractional inequality on time scales]
\label{thm:Mink_ineq}
Let $\alpha\in (0, 1]$, $a, b\in\mathbb{T}$ and $p>1$.
If $f, g, h:[a, b]\rightarrow\mathbb{R}$ are $rd$-continuous, then
\begin{multline}
\label{eq:Mink_ineq}
\Bigg[\int_a^b |(f+g)(t)|^p |h(t)| \Delta^{\alpha}t\Bigg]^{1/p}\\
\leq \Bigg[\int_a^b |f(t)|^p |h(t)| \Delta^{\alpha}t\Bigg]^{\frac{1}{p}}
+\Bigg[\int_a^b |g(t)|^p |h(t)| \Delta^{\alpha}t\Bigg]^{\frac{1}{p}}.
\end{multline}
\end{thm}

\begin{proof}
We apply H\"{o}lder's inequality \eqref{eq:Hineq} with $q=p/(p-1)$ 
and items (i) and (v) of Theorem~\ref{Int-Proprty} to obtain
\begin{align*}
\int_a^b &|(f+g)(t)|^p |h(t)| \Delta^{\alpha}t\\
&=\int_a^b |(f+g)(t)|^{p-1}|(f+g)(t)| |h(t)| \Delta^{\alpha}t\\
&\leq \int_a^b |f(t)||(f+g)(t)|^{p-1} |h(t)| \Delta^{\alpha}t
+ \int_a^b |g(t)||(f+g)(t)|^{p-1} |h(t)| \Delta^{\alpha}t\\
&\leq  \Bigg[\int_a^b |f(t)|^p |h(t)| \Delta^{\alpha}t\Bigg]^{\frac{1}{p}}
\Bigg[\int_a^b |(f+g)(t)|^{(p-1)q} |h(t)| \Delta^{\alpha}t\Bigg]^{\frac{1}{q}}\\
&\quad +  \Bigg[\int_a^b |g(t)|^p |h(t)| \Delta^{\alpha}t\Bigg]^{\frac{1}{p}}
\Bigg[\int_a^b |(f+g)(t)|^{(p-1)q} |h(t)| \Delta^{\alpha}t\Bigg]^{\frac{1}{q}}\\
&=\Bigg[\int_a^b |(f+g)(t)|^p |h(t)| \Delta^{\alpha}t\Bigg]^{\frac{1}{q}}\\
&\qquad \times \Bigg(\Bigg[\int_a^b |f(t)|^p |h(t)| \Delta^{\alpha}t\Bigg]^{\frac{1}{p}}
+ \Bigg[\int_a^b |g(t)|^p |h(t)| \Delta^{\alpha}t\Bigg]^{\frac{1}{p}}\Bigg).
\end{align*}
Dividing both sides of the obtained inequality by
$\Bigg[\int_a^b |(f+g)(t)|^p |h(t)|\Delta^{\alpha}t\Bigg]^{\frac{1}{q}}$,
we arrive at the Minkowski inequality \eqref{eq:Mink_ineq}.
\end{proof}

Jensen's classical inequality relates the value of a convex/concave 
function of an integral to the integral 
of the convex/concave function. We prove a generalization
of such relation for the BHT fractional calculus on time scales.

\begin{thm}[Generalized Jensen's fractional inequality on time scales]
\label{thm:Jen} 
Let $\mathbb{T}$ be a time scale, $a, b\in\mathbb{T}$ with $a<b$,
$c, d\in\mathbb{R}$, $\alpha \in (0, 1]$, 
$g \in C\left([a,b]\cap\mathbb{T}; (c,d)\right)$
and $h \in C\left([a,b]\cap\mathbb{T}; \mathbb{R}\right)$ with
$$
\int_{a}^{b} |h(s)| \Delta^\alpha s > 0.
$$
\begin{itemize}
\item If $f \in C\left((c,d); \mathbb{R}\right)$ is convex, then
\begin{equation}
\label{eq:Jen:conv}
f\left(\frac{\int_{a}^{b} g(s) |h(s)| \Delta^\alpha s}{\int_{a}^{b} |h(s)| \Delta^\alpha s}\right)
\leq \frac{\int_{a}^{b} f(g(s)) |h(s)| \Delta^\alpha s}{\int_{a}^{b} |h(s)| \Delta^\alpha s}.
\end{equation}

\item If $f \in C\left((c,d); \mathbb{R}\right)$ is concave, then
\begin{equation}
\label{eq:Jen:conc}
f\left(\frac{\int_{a}^{b} g(s) |h(s)| \Delta^\alpha s}{\int_{a}^{b} |h(s)| \Delta^\alpha s}\right)
\geq \frac{\int_{a}^{b} f(g(s)) |h(s)| \Delta^\alpha s}{\int_{a}^{b} |h(s)| \Delta^\alpha s}.
\end{equation}
\end{itemize}
\end{thm}

\begin{proof}
We start by proving \eqref{eq:Jen:conv}. Since $f$ is convex,
for any $t \in (c, d)$ there exists $a_t \in \mathbb{R}$ such that
\begin{equation}
\label{eq:ineq:conv}
a_t (x-t) \leq f(x) - f(t) \quad \text{ for all } x \in (c,d).
\end{equation}
Let
$$
t = \frac{\int_{a}^{b} g(s) |h(s)| \Delta^\alpha s}{\int_{a}^{b} |h(s)| \Delta^\alpha s}.
$$
It follows from \eqref{eq:ineq:conv} and item (v) of Theorem~\ref{Int-Proprty} that
\begin{equation*}
\begin{split}
\int_{a}^{b} f(g(s)) & |h(s)| \Delta^\alpha s
-\left(\int_{a}^{b} |h(s)| \Delta^\alpha s\right) 
f\left(\frac{\int_{a}^{b} g(s) |h(s)| \Delta^\alpha s}{\int_{a}^{b} |h(s)| \Delta^\alpha s}\right)\\
&= \int_{a}^{b} f(g(s)) |h(s)| \Delta^\alpha s
-\left(\int_{a}^{b} |h(s)| \Delta^\alpha s\right) f\left(t\right)\\
&= \int_{a}^{b} \left(f(g(s))-f(t)\right) |h(s)| \Delta^\alpha s\\
\end{split}
\end{equation*}
\begin{equation*}
\begin{split}
&\geq a_t \int_{a}^{b} \left(g(s)-t\right) |h(s)| \Delta^\alpha s\\
&= a_t \left(\int_{a}^{b} g(s) |h(s)| \Delta^\alpha s
- t \int_{a}^{b} |h(s)| \Delta^\alpha s\right)\\
&= a_t \left(\int_{a}^{b} g(s) |h(s)| \Delta^\alpha s
- \int_{a}^{b} g(s) |h(s)| \Delta^\alpha s\right)\\
&= 0.
\end{split}
\end{equation*}
This proves \eqref{eq:Jen:conv}.
To prove \eqref{eq:Jen:conc}, we simply observe that $F(x) = -f(x)$ is convex
(because we are now assuming $f$ to be concave) and then we apply 
inequality \eqref{eq:Jen:conv} to function $F$.
\end{proof}

We end with an application of Theorem~\ref{thm:Jen}.

\begin{thm}[A weighted fractional Hermite--Hadamard inequality on time scales]
\label{thm:HHineq}
Let $\mathbb{T}$ be a time scale, $a, b \in \mathbb{T}$ and $\alpha \in (0,1]$.
Let $f : [a,b] \rightarrow \mathbb{R}$ be a continuous convex function and let 
$w : \mathbb{T} \rightarrow \mathbb{R}$ be a continuous function such that
$w(t) \geq 0$ for all $t \in \mathbb{T}$ and $\int_a^b w(t) \Delta^\alpha t > 0$.
Then,
\begin{equation}
\label{eq:HHineq}
f(x_{w,\alpha}) \leq 
\frac{1}{\int_a^b w(t) \Delta^\alpha t} \int_{a}^{b} f(t) w(t) \Delta^\alpha t
\leq \frac{b - x_{w,\alpha}}{b-a} f(a) + \frac{x_{w,\alpha} - a}{b-a} f(b),
\end{equation}
where $x_{w,\alpha} = \frac{\int_{a}^{b} t w(t) \Delta^\alpha t}{\int_{a}^{b} w(t) \Delta^\alpha t}$.
\end{thm}

\begin{proof}
For every convex function one has
$$
f(t) \leq f(a) + \frac{f(b)-f(a)}{b-a} (t-a).
$$
Multiplying this inequality with $w(t)$, which is nonnegative, we get
$$
w(t) f(t) \leq f(a) w(t) + \frac{f(b)-f(a)}{b-a} (t-a) w(t).
$$
Taking the $\alpha$-fractional integral on both sides, we can write that
$$
\int_{a}^{b} w(t) f(t) \Delta^\alpha t 
\leq \int_{a}^{b} f(a) w(t) \Delta^\alpha t
+ \int_{a}^{b} \frac{f(b)-f(a)}{b-a} (t-a) w(t) \Delta^\alpha t,
$$ 
which implies
\begin{multline*}
\int_{a}^{b} w(t) f(t) \Delta^\alpha t \\
\leq f(a) \int_{a}^{b} w(t) \Delta^\alpha t
+ \frac{f(b)-f(a)}{b-a} \left(\int_{a}^{b}  t w(t) \Delta^\alpha t 
- a \int_{a}^{b} w(t) \Delta^\alpha t\right),
\end{multline*}
that is,
$$
\frac{1}{\int_a^b w(t) \Delta^\alpha t} \int_{a}^{b} f(t) w(t) \Delta^\alpha t
\leq \frac{b - x_{w,\alpha}}{b-a} f(a) + \frac{x_{w,\alpha} - a}{b-a} f(b).
$$
We have just proved the second inequality of \eqref{eq:HHineq}.
For the first inequality of \eqref{eq:HHineq}, we use \eqref{eq:Jen:conv}
of Theorem~\ref{thm:Jen} by taking $g : \mathbb{T} \rightarrow \mathbb{T}$
defined by $g(s) = s$ for all $s \in \mathbb{T}$ and 
$h : \mathbb{T} \rightarrow \mathbb{R}$ given by $h = w$. 
\end{proof}

Note that if in Theorem~\ref{thm:HHineq} we consider a concave function $f$
instead of a convex one, then the inequalities of \eqref{eq:HHineq} are reversed.


\section*{Acknowledgements}

Torres was partially supported by the Portuguese Foundation for Science
and Technology (FCT), through the Center for Research and Development
in Mathematics and Applications (CIDMA), within project UID/MAT/04106/2013.
The authors are greatly indebted to two referees for their several useful
suggestions and valuable comments.



\end{document}